%% file: prvSpherangmin7.tex
\documentclass[11pt,a4paper]{article}

\usepackage{tikz}
\usepackage{amsfonts,amssymb,amsmath, amsthm}
\usepackage[all]{xy}
\usepackage{graphicx,color}
\usepackage{srcltx}
\usepackage[latin1]{inputenc}  
\usepackage[T1]{fontenc}
\usepackage{subeqnarray}

\def\Orb{{\mathcal O}}

\def\hw{{\hat w}}
\def\hv{{\hat v}}
\def\hG{{\hat G}}

\def\hs{{\hat s}}
\def\hnu{{\hat \nu}}

\def\hT{{\hat T}}
\def\hB{{\hat B}}

\def\hW{{\hat W}}
\def\halpha{{\hat \alpha}}

\def\longto{\longrightarrow}

\def\PP{{\mathbb P}}
\def\QQ{{\mathbb Q}}
\def\CC{{\mathbb C}}

\def\Spin{{\rm Spin}}

\def\l{\langle}
\def\r{\rangle}

\def\SL{{\rm SL}}

\def\Li{{\mathcal{L}}}
\def\Mi{{\mathcal{M}}}

\def\Qleft{\hspace*{-1pt}\setminus\hspace*{-3pt}}

\newtheorem{lemma}{Lemma}
\newtheorem{prop}{Proposition}
\newtheorem{theo}{Theorem}
\newtheorem{coro}{Corollary}

\newtheorem{ex}{Example}

\newtheorem{remark}{Remark}

\begin{document}
\title{Generalizations of the PRV conjecture, II}
\author{PL. Montagard \& B. Pasquier \& N. Ressayre\footnote{N.R. was partially supported by the French National Research Agency (ANR-09-JCJC-0102-01).}}



\maketitle

\begin{abstract}
Let $G\subset\hat{G}$ be two complex connected reductive groups. We deals with the hard problem of finding sub-$G$-modules of a given irreducible $\hat{G}$-module. In the case where $G$ is diagonally embedded in $\hat{G}=G\times G$, S.~Kumar and O.~Mathieu found some of them, proving the PRV conjecture. Recently, the authors generalized the PRV conjecture on the one hand to the case where $\hat{G}/G$ is spherical of minimal rank, and on the other hand giving more sub-$G$-modules in the classical case $G\subset G\times G$. In this paper, these two recent generalizations are combined in a same more general result.
\end{abstract}

\textbf{Mathematics Subject Classification} 22E46 17B10 14L24\\

\textbf{Keywords} Branching rules, Affine spherical homogeneous spaces of minimal rank, Tensor product decomposition, PRV conjecture.\\

\section{Introduction}

Let $G$ be a complex connected reductive group.
Let $T$ be a maximal torus of $G$ and let $B$ be a Borel subgroup
containing $T$. The  Weyl
group is denoted by $W$.
The irreducible $G$-module of highest weight $\mu$ is denoted by
$V_G(\mu)$.
For any weight $\mu$ of $T$, there exists a unique dominant weight
$\overline{\mu}$ in the $W$-orbit of $\mu$. The representation
$V_G(\bar\mu)$ is called the {\it irreducible $G$-module with extremal
  weight $\mu$}. 
Parthasarathy-Ranga Rao-Varadarajan conjectured in the sixties the following
statement.

\bigskip
\noindent{\it The PRV conjecture.}

Let $\mu$ and $\nu$ be two dominant weights.
Then, for any $w\in W$, the irreducible $G$-module $V_G(\overline{\mu+w\nu})$ 
with extremal weight
$\mu+w\nu$, occurs with multiplicity at least one in $V_G(\mu)\otimes V_G(\nu)$.\\

This conjecture was proved independently by S.~Kumar in \cite{Kumar:prv1} and
O.~Mathieu in \cite{Mathieu:prv}.
This paper is a continuation of  \cite{MPR} where  the authors obtained a new proof and two
generalizations of the PRV conjecture. 


\bigskip
Assume  that $G$ is a subgroup of a bigger connected reductive group $\hat G$.
The decomposition of the tensor product concerned by the PRV
conjecture is an example ({\it i.e.} when $\hG=G\times G$ and $G$ is
diagonally embedded in $\hG$) of the following
problem.

\begin{center}\it
  Find irreducible $G$-submodules in a given  irreducible $\hG$-module.
\end{center}

\bigskip
The homogeneous space $(G\times G)/G$ is  spherical  
of minimal rank (see {\it e.g.} \cite{spherangmin} for a precise definition).
From now on, we consider  two reductive connected
groups $G\subset\hG$ such that $\hG/G$ is spherical 
of minimal rank. Let $\hT$ be the centralizer of $T$ in $\hG$,
then it is well-know that $\hT$ is a maximal torus of $\hG$ (see Lemma~\ref{lemma:centra} in the appendix). We can deduce that the Weyl group $W$ of $G$ is
canonically a subgroup of the Weyl group $\hW$ of $\hG$ (see Corollary~\ref{coro:weyl} in the appendix). 

In \cite{MPR}, we generalized the PRV conjecture:
for any irreducible  $\hG$-module $\hat V$, we give an explicit family of irreducible $G$-submodules of
$\hat V$ parametrized by $W\Qleft\hW$. 

Let $\mu$ and $\nu$ be two dominant weights of $G$.
In \cite{MPR},  we also improve the
PRV conjecture by constructing a  family of irreducible
submodules of $V_G(\mu)\otimes V_G(\nu)$,  strictly  containing the PRV components.
The cardinality of this family is not bounded independently of the
couple $(\lambda,\mu)$.
More precisely, this family is the set of dominant weights
belonging to some explicit line segments with (at least) one PRV
component as end point.

The main result of this paper gives a  similar family of
irreducible sub-$G$-modules of any irreducible $\hG$-module $\hat V$.
In particular, this result is a common generalization of the two
results of \cite{MPR} explained before.

\bigskip

The precise statement needs some preparation.
If $\alpha$ is a root of $(G,T)$, $\alpha^\vee$ denotes the
corresponding coroot.
There exists a unique   Borel subgroup $\hat B$ containing $\hT$ such
that $\hat B\cap G=B$ (see for example \cite[Proposition~2.2]{spherangmin}).
Consider the restriction map $\rho:\,X(\hat T)\longto X(T)$ from the character
group of $ \hat T$ to this of $T$.
Let  $X(T)^+$ (resp. $X(\hat T)^+$) denote the set of dominant weights
of $G$ (resp. of $\hG$).
Let $\hat\nu\in X(\hat T)^+$;
in \cite{MPR}, we proved that for any $\hw\in \hW$,
$V_G(\overline{\rho(\hw\hnu)})$ is a submodule of $V_{\hG}(\hnu)$. 
Such a submodule of $V_{\hG}(\hnu)$ is called a {\it PRV component}. 
Recall that $W$ is canonically a subgroup of $\hW$ and 
observe  that, for any $w$ in $W$,
$\overline{\rho(\hw\hnu)}=\overline{\rho(w\hw\hnu)}$. 
In particular, the  PRV components are parametrized by
$W\Qleft\hW$.

Let $\Delta$ and  $\hat \Delta$ denote the set of simple roots of
$G$ and $\hG$.
By \cite[Lemma~4.6]{spherangmin}, $\rho(\hat \Delta)=\Delta$. More precisely, for any 
$\alpha\in\Delta$, we have the following alternative:
\begin{enumerate}
\item there exists a unique $\halpha_0\in\hat\Delta$ such that
$\rho(\halpha_0)=\alpha$; or
\item there exist exactly two  simple roots $\halpha_1$ and
  $\halpha_2$ in $\hat\Delta$ such that
$\rho(\halpha_1)=\rho(\halpha_2)=\alpha$.
Moreover $\halpha_1$ and $\halpha_2$ are orthogonal.
\end{enumerate}
The set of simple roots satisfying the first condition is
denoted by $\Delta_1$ and the set of those satisfying the second one is denoted
by $\Delta_2$.

\begin{theo}\label{th:main}
Recall that $G\subset \hG$ are connected reductive groups and that
$\hG/G$ is assumed to be spherical of
minimal rank. 
Let $\hnu$ be a dominant weight of $\hG$ and let $\hw\in\hW$ such that
$\rho(\hw\hnu)$ is a dominant weight of $G$. 
Let $\alpha\in\Delta_2$. Index the two simple roots $\halpha_1$ and
$\halpha_2$ in $\hat \Delta$ belonging to the pullback of $\alpha$ by
$\rho$ in  such a way that $\langle
\hw\hnu,\,\halpha_1^\vee\rangle\leq \langle
\hw\hnu,\,\halpha_2^\vee\rangle$. Denote by $S_{\alpha,\hw,\hnu}$ the
line segment in
$X(T)\otimes\QQ$ whose end points are $\rho(\hw\hnu)$ and
$\rho(s_{\halpha_j}\hw\hnu)$. Then for any $\nu\in
S_{\alpha,\hw,\hnu}\cap X(T)^+$ the irreducible $G$-module $V_G(\nu)$ occurs with
multiplicity at least one in $V_\hG(\hnu)$.
\end{theo}

From each PRV component in $V_\hG(\hnu)$, Theorem~\ref{th:main} gives
$|\Delta_2|$ line segments (that may be of length zero) of irreducible components of $V_\hG(\hnu)$.
Note that the length of $S_{\alpha,\hw,\hnu}$ is
$|\langle
\hw\hnu,\,\halpha_1^\vee\rangle|$.

\bigskip

The proof of Theorem~\ref{th:main} above is the object of Section \ref{sec:proof}. 
In Section~\ref{sec:G2}, Theorem~\ref{th:main} is illustrated with the
example of $G=G_2$ in $\hG=\Spin_7$.

In Section~\ref{sec:ortho}, we get a still larger family of
irreducible $G$-submodules of $\hat V$.
In Section~\ref{sec:sym}, we get a still larger family of
irreducible $G$-submodules of $V_G(\lambda)\otimes V_G(\mu)$. This
last result is specific to the case of the tensor product decomposition. These two
results are illustrated by examples. 
We do not detail here the results of these two sections, whose statements are quiet technical, we refer the reader to Theorems \ref{th:orth}, \ref{th:mainsym} and \ref{th:orthsym}.

\section{Proof of Theorem~\ref{th:main}}\label{sec:proof}

\subsection{A reduction}\label{sec:reduction}

In this section we rewrite Theorem~\ref{th:main} as it was written in the tensor product case in \cite{MPR}. In fact, we prove that the following theorem implies Theorem~\ref{th:main}.

\begin{theo}\label{th:tech}
  Let $\nu$ and $\hnu$ be two dominant weights of $T$ and $\hT$.
Assume that there exist a simple root $\alpha\in\Delta_2$, $\hw$
in the Weyl group $\hW$ and an integer $k$ such that
\begin{eqnarray}
  \label{eq:line}
  \nu=\rho(\hw\hnu)-k\alpha.
\end{eqnarray}

Denote by $\halpha_1$ and $\halpha_2$ the two roots in $\hat{\Delta}$ belonging to the pullback of $\alpha$ by $\rho$.

If  
\begin{subeqnarray}
  \label{eq:1}
  k & \geq & 0,\\
k &\leq & \langle \hw\hnu,\,\halpha_1^\vee\rangle,\\
k & \leq & \langle \hw\hnu,\,\halpha_2^\vee\rangle,
\end{subeqnarray}
 then the irreducible $G$-module $V_G(\nu)$ occurs with multiplicity at
least one in $V_\hG(\hnu)$.
\end{theo}

Before  proving that Theorem~\ref{th:tech} implies
Theorem~\ref{th:main}, we state a basic remark.

\begin{remark}\label{rq:useful}
Let $\alpha\in
\Delta_2$. 
If $S$ is the neutral component of the kernel of $\alpha$, \cite[Lemmas 4.3 and 4.4]{spherangmin} shows that
the inclusion of $G^S$ in $\hG^S$ is isomorphic to the diagonal
inclusion of $(P)SL_2$ in  $(P)SL_2\times (P)SL_2$.
It follows that the image of the one-parameter subgroup
$\halpha^\vee_1+\halpha^\vee_2$ of $\hT$ is
contained in $T$, and that
$\alpha^\vee=\halpha^\vee_1+\halpha^\vee_2$. 

\end{remark}

\begin{proof}[Proof of: Theorem~\ref{th:tech} implies
    Theorem~\ref{th:main}.]

Let $\hnu$ be a dominant weight of $\hG$, $\hw\in\hW$ such that
$\rho(\hw\hnu)$ is a dominant weight of $G$, $\alpha\in\Delta_2$ and
$\nu\in S_{\alpha,\hw,\hnu}\cap X^+(T)$.
  Recall that $\langle
\hw\hnu,\,\halpha_1^\vee\rangle\leq \langle
\hw\hnu,\,\halpha_2^\vee\rangle$.  Two different cases occur
according to the sign of $\langle \hw\hnu,\,\halpha_1^\vee\rangle$. 
\begin{enumerate}
\item If $\langle \hw\hnu,\,\halpha_1^\vee\rangle$ is non-negative, then
  $\langle \hw\hnu,\,\halpha_2^\vee\rangle$ is also non-negative. 
Denoting by $k$ the integer such
  that $\nu=\rho(\hw\hnu)-k\alpha$,  the  definition of
  $S_{\alpha,\hw,\hnu}$ implies that $k$ satisfies  inequalities (\ref{eq:1}).
Then, by Theorem \ref{th:tech}, the irreducible module
  $V_G(\nu)$ occurs in $V_\hG(\hnu)$.

\item Suppose that $\langle \hw\hnu,\,\halpha_1^\vee\rangle$ is
  negative, so that $\langle
  s_{\halpha_1}\hw\hnu,\,\halpha_1^\vee\rangle\geq 0$.

 By  Remark~\ref{rq:useful}, $\langle
\hw\hnu,\,\halpha_1^\vee+\halpha_2^\vee\rangle\geq 0$. Then $\langle
\hw\hnu,\,\halpha_2^\vee\rangle\geq 0$.  Since $\halpha_1$ and
  $\halpha_2$ are orthogonal, we deduce that $\langle
  s_{\halpha_1}\hw\hnu,\,\halpha_2^\vee\rangle\geq 0$.
  
Denote by $k$ the integer such that
$\nu=\rho(s_{\halpha_1}\hw\hnu)-k\alpha$. Then $k$ satisfies
inequalities (\ref{eq:1}) with $s_{\halpha_1}\hw$ instead of $\hw$. Hence the
irreducible module $V_G(\nu)$ occurs in $V_\hG(\hnu)$.
\end{enumerate}
\end{proof}

\begin{remark}
   In the first case, the two end points of the line segment
$S_{\alpha,\hw,\hnu}$ are dominant and both correspond to a PRV
component.
In the second one,   $\rho(s_{\halpha_1}\hw\hnu)$ is not necessarily dominant.
\end{remark}

\subsection{Proof of Theorem~\ref{th:tech}}\label{sec:proof:th:tech}
\subsubsection{Preliminaries}\label{sec:preliminary}

The existence of $k$ satisfying inequalities~\eqref{eq:1} implies that $\langle
\hw\hnu,\,\halpha_1^\vee\rangle$ and
$\langle\hw\hnu,\,\halpha_2^\vee\rangle$ are non-negative. If one of
them is zero, $\nu$ is a PRV component, and in that case,
Theorem~\ref{th:tech} is a consequence of our previous
work~\cite[Theorem~1]{MPR}.
Hence, we assume from now that $\langle
\hw\hnu,\,\halpha_1^\vee\rangle$ and
$\langle\hw\hnu,\,\halpha_2^\vee\rangle$ are positive.   
In particular, by Lemma~\ref{lemma:length} in Appendix, we have  $l(s_{\halpha_1}
\hw)=l(s_{\halpha_2}\hw)=l(\hw)+1$. 

\subsubsection{An asymptotic result}\label{sec:asymptotic}
In this section we apply the Borel-Weil theorem and an argument of
Geometric Invariant Theory to get an asymptotic version of
Theorem~\ref{th:tech}. 

For any character $\nu$ of $B$, $\Li_\nu$ denote the $G$-linearized
line bundle on $G/B$ such that $B$ acts with weight $-\nu$ on the
fiber over $B/B$. Similarly, we define the line bundle $\Li_\hnu$ on
$\hG/\hB$. The Borel-Weil theorem asserts that
$H^0(G/B,\Li_\nu)$ is the irreducible $G$-module $V_G(\nu)^*$.

Set $X=G/B\times \hG/\hB$ and
$\hv=s_{\halpha_1}.\hw$. Let $S$ be the neutral component of the
Kernel of $\alpha$. Let $\rho_S$ denote the restriction map to $S$ both for characters of
$T$ and $\hT$. We denote by $G^S$
(resp. by $\hG^S$) the centralizer of $S$ in $G$ (resp. in $\hG$). 

Let $w_0$ denote the longest element in $W$ and consider the following irreducible component of the $S$-fixed points
set:
$$
C=G^Sw_0 B/B\times \hG^S\hw\hB/\hB.
$$

Note that  $C$ is isomorphic to $(\PP^1)^3$ (see for example \cite[Chapter~9]{Hum}).
Finally consider the line bundle $\Mi=\Li_{-w_0\nu}\boxtimes\Li_{\hnu}$
on $X$. Then, for any
non-negative integer $n$,   
$H^0(X,\Mi^{\otimes n})=V_G(n\nu)\otimes V_\hG(n\hnu)^*$.

The torus $S$ acts trivially on $C$ and hence on the restriction
$\Mi_{|C}$ of $\Mi$ on $C$ by a
character. This character equals  $\rho_S(\nu)+\rho_S(-\hw\hnu)$.
By assumption~\eqref{eq:line}, this character is trivial. In particular,
any point of $C$ is semistable for the action of $S$.

The restriction of $\Mi$ to $C\simeq (\PP^1)^3$ is isomorphic as a line
bundle to 
$\Orb(a)\otimes\Orb(b)\otimes\Orb(c)$ for some integers $a,b,$ and
$c$.
The torus   $T$ acts on the fiber over $w_0 B/B$ in the line bundle
$\Li_{-w_0\nu}$ by the weight $\nu$. 
The $\SL_2$-theory of $\PP^1$ implies that $a=\langle \nu,\,
\alpha^\vee\rangle$.
Similarly we compute $b$ and $c$ and we get
$$
\begin{array}{l}
  a=\langle \nu,\, \alpha^\vee\rangle,\\
b=\langle \hw\hnu,\,\halpha_1^\vee \rangle,\\
c=\langle \hw\hnu,\,\halpha_2^\vee \rangle.
\end{array}
$$

 Then, using Remark \ref{rq:useful}, inequalities~(\ref{eq:1}) are
 equivalent to 
$$
\begin{array}{l}
  a+b\geq c,\\
a+c\geq b,\\
b+c\geq a.
\end{array}
$$ 
They imply that
$C^{\rm ss}(\Mi,G^S/S)$ is not empty.

Since $C^{\rm ss}(\Mi,G^S)\neq\varnothing$,
  Luna's theorem (see \cite[Corollary~2 and Remark~1]{Luna:adh}
or see also \cite[Proposition~8]{GITEigen}) implies that $X^{\rm
  ss}(\Mi)\cap C\neq\varnothing$. 
Fix a positive integer $n$ and a $G$-invariant section $\sigma$ in
$H^0(X,\Mi^{\otimes n})$ such that $\sigma_{|C}$ is not identically
zero.
Let $X_\hv$ be the closure of the orbit $G\hv\hB/\hB$ in $\hG/\hB$.
By the remark of Section \ref{sec:preliminary}, the stabilizer of
$\hv$ in $G^S$ equals $T$. Then 
$G^S.\hv$ has dimension two and $\overline{G^S\hv}=\hG^S\hv$.
Hence
$C\cap(\{w_0 B/B\}\times \hG/\hB)$ is contained in $\{w_0
B/B\}\times X_\hv$.  

Consider now the restriction $\tau$ of $\sigma$ to 
$Y=\{w_0 B/B\}\times \hG/\hB$. Identifying $Y$ with the $\hG/\hB$ by the
second projection, $\Mi_{|Y}$ is a $B^-$-linearized line bundle on
$\hG/\hB$. More precisely, it is the line bundle $\Li_\hnu$ where the
action of $B^-$ is obtained by restricting the action of $\hG$ and
then by twisting by $\nu$.
 But $\tau$ is a $B^-$-invariant section of
$\Mi^{\otimes n}_{|Y}$. Then $\tau$ identifies with a section $\bar\tau$ of $\Li_{n\hnu}$
on $\hG/\hB$ which is a $B^-$-eigenvector of weight $-n\nu$. 
Moreover, the relative position of $C$ and $X_\hv$ implies that
$\bar\tau_{|X_\hv}$ is not identically zero. Note that this implies that
$V_G(n\nu)^*$ appears in $V_\hG(n\hnu)^*$.\\

\subsubsection{Multiplicity one}\label{sec:mult_one}

We introduce a notation: let $H$ be an algebraic group and let $\chi$
be a character of $H$; if $V$ is a representation of $H$, we denote by
$V^{(H)_{\chi}}$ the subspace of $H$-eigenvectors of weight $\chi$. 

Define $X_\hv^\circ=G\hv\hB/\hB$, and recall that
$X_\hv=\overline{X_\hv^\circ}.$
Now, we prove that the multiplicity of $V_G(\nu)^*$ in
$H^0(X^\circ_\hv,\Li_\hnu)$ is one.
We have an injection
$$\iota\ :\ H^0(X^\circ_\hv,\Li_\hnu)\rightarrow \CC[G]$$
and the image $\iota(H^0(X^\circ_\hv,\Li_\hnu))$ equals $\CC[G]^{(\{1\}\times G_\hv)_{\rho(\hv\hnu)}}$. The isotropy subgroup
$G_\hv$ contains the torus T. Hence by using the Frobenius theorem, we
can deduce the following:  
$$
\iota(H^0(X^\circ_\hv,\Li_\hnu))\subset\CC[G]^{(\{1\}\times T)_{\rho(\hv
  \hnu)}}\simeq\bigoplus_{\chi\in X(T)^+}V_G(\chi)^*\otimes
V_G(\chi)^{(T)_{\rho(\hv \hnu)}}.
$$

In particular the multiplicity of $V_G(\nu)^*$ in
$H^0(X^\circ_\hv,\Li_\hnu)$ is at most the dimension of $V_G(\nu)^{(T)_{\rho(\hv \hnu)}}$.
But, by assumption, 
$$
\rho(\hv\hnu)=\rho(\hw\hnu-\langle\hw\hnu,\,\halpha^\vee_1\rangle\halpha^\vee_1)=
\nu-l\alpha,
$$
where $l=\langle\hw\hnu,\,\halpha^\vee_1\rangle-k$. 
Then the dimension of  $V_G(\nu)^{(T)_{\rho(\hv \hnu)}}$ is at most one and the
multiplicity of $V_G(\nu)^*$ in $H^0(X^\circ_\hv,\Li_\hnu)$ is at most
one.

We now show that this multiplicity is one. 
Recall that 
$\alpha^\vee=\halpha^\vee_1+\halpha^\vee_2$ (see Remark \ref{rq:useful}). 
Hence inequalities~\eqref{eq:1} implies that $0\leq l\leq \langle
\nu,\,\alpha^\vee\rangle$.
Therefore the dimension of $V_G(\nu)^{(T)_{\rho(\hv \hnu)}}$ is one.
Let us choose a non-zero element $f$ in the space
$\CC[G]^{(B_-\times T)_{(-\nu,\rho(n\hv\hnu))}}$. The subgroup
$G_\hv$ is a  solvable subgroup
of $G$ containing $T$, so
$G_\hv=G_\hv^u.T$, where $G_\hv^u$ is the  unipotent radical of
$G_\hv$. It is sufficient to prove that
$f$ is $G_\hv^u$-invariant.

Consider the element $\iota(\bar\tau)$: it is a non-zero
element belonging to the one dimensional space:
$\CC[G]^{(B_-\times T_\hv)_{(-n\nu,\rho(n\hv\hnu))}}$. Hence, 
up to multiplying  $f$ by a non-zero
scalar, we may assume that $f^n=\iota(\bar\tau)$. 
We deduce that $f^n$ is $G^u_\hv$-invariant. Hence $f$ is
$G^u_\hv$-invariant because $G^u_\hv$ possesses no non-trivial character and the
algebra $\CC[G]$ is factorial.

\subsubsection{Conclusion}\label{sec:conclusion}
We can now finish the proof of Theorem~\ref{th:tech} by showing that
$V_G(\nu)^*\subset V_\hG(\hnu)^*$. 
Consider the following morphisms: 
$$
\xymatrix{
{\rm H}^0(\hG/\hB,\Li_\hnu)
\ar@{->>}_\varphi[d]\\
{\rm H}^0(X_\hv,\Li_\hnu)\ar@{^{(}->}_\pi[d]\\
{\rm H}^0({X^{\circ}_{\hv_{}}},\Li_{\hnu})\ar@{^{(}->>}_\iota[d]\\
\CC[G]\\
}
$$

The morphisms $\pi$ and $\iota$ are defined above and $\varphi$ is the
restriction morphism.
By \cite[Corollary~8]{Br:GammaGH}, the map $\varphi$ is surjective.
Since $ H^0(\hG/\hB,\Li_\hnu)$ is isomorphic to $V_\hG(\hnu)^*$, it is
sufficient to prove that $V_G(\nu)^*\subset H^0(X_\hv,\Li_\hnu)$.

In Section \ref{sec:mult_one}, we showed that there exists a
non-zero  $B^-$-equivariant
section $\kappa$  in $H^0(X_\hv^\circ,\Li_\hnu)$ of weight $-\nu$.
It remains to prove that $\kappa$ extends to a regular section on
$X_\hv$ to conclude the proof.
Since the space $H^0(X^\circ_\hv,\Li_{n\hnu})^{(B^-)_{n\nu}}$ is one dimensional,
we may assume that $\kappa^{\otimes n}=\bar\tau$, where $\bar\tau$ is
the section obtained in \ref{sec:asymptotic}, so $\kappa^{\otimes n}$
extends to a regular section on $X_\hv$.
The normality of $X_\hv$ (see \cite[Theorem~1]{Br:multfree}) implies
that  $\kappa$ extends to a regular section on $X_\hv$.

\begin{remark}
In general the multiplicity of $V_G(\nu)$ in
$V_\hG(\hnu)$ does not equal one; but, as a consequence of the proof,
the multiplicity of $V_G(\nu)^\ast$ in 
${\rm H}^0(X_\hv,\Li_\hnu)$ is one.
\end{remark}

\subsection{The example $G_2\subset{\rm Spin}_7$}
\label{sec:G2}

In this section $\hG={\rm Spin}_7$ and $G=G_2$. 
The first fundamental representation of $G_2$ has dimension 7 and
induces an embedding of $G_2$ in ${\rm SO}_7$. Since $G_2$ is simply
connected this embedding can be raised to $\hG={\rm Spin}_7$.
We use the numeration of Bourbaki \cite{Bou}. 
Denote by $\hs_1$, $\hs_2$ and $\hs_3$ the simple reflections of
$\hW$. 
The fundamental weights of $\hG$ and $G$ are denoted by
$\hat{\varpi_i}$ for $i=1,\,2,\,3$ and by $\varpi_i$ for $i=1,\,2$.
Here $\rho$ is characterized by: 
$\rho(\hat{\varpi_1})=\varpi_1$, $\rho(\hat{\varpi_2})=\varpi_2$ and
$\rho(\hat{\varpi_3})=\varpi_1$.
We use the basis of fundamental weights for $\hG$ and $G$ to express the weights; in
particular $\hnu=(\hnu_1,\hnu_2,\hnu_3)$ means $\hnu=\hnu_1\varpi_1+\hnu_2\varpi_2+\hnu_3\varpi_3$.

We begin by giving the  PRV components of any irreducible representation of $\hG$.

\begin{prop}
Let $\hnu=(\hnu_1,\hnu_2,\hnu_3)$ be a dominant weight of $\hG$ (ie
$\hnu_i\geq 0$ for any $i$).
Then the dominant weights of $G$, obtained as the restriction $\rho(\hw\hnu)$ of an element of the orbit $\hW.\hnu$, are
\begin{itemize}
\item $\nu^1=(\hnu_1+\hnu_3,\hnu_2)$;
\item $\nu^2=(-\hnu_1+\hnu_3,\hnu_1+\hnu_2)$ if $\hnu_1\leq\hnu_3$\\
or $\nu^2=(\hnu_1-\hnu_3,\hnu_2+\hnu_3)$ if $\hnu_1\geq\hnu_3$;
\item $\nu^3=(-\hnu_1+\hnu_2+\hnu_3,\hnu_1)$ if $\hnu_1\leq\hnu_2+\hnu_3$\\
or $\nu^3=(\hnu_1-\hnu_2-\hnu_3,\hnu_2+\hnu_3)$ if $\hnu_1\geq\hnu_2+\hnu_3$;
\item $\nu^4=(-\hnu_1+\hnu_2,\hnu_1)$ if $\hnu_1\leq\hnu_2$\\
or $\nu^4=(\hnu_1-\hnu_2,\hnu_2)$ if $\hnu_1\geq\hnu_2$.
\end{itemize}
\end{prop}
It may happen of course that some of the four  PRV components above are the same.

\begin{proof}
First remark that, since $|\hW/W|=4$ we have at most 4  PRV
components. The weights $\nu^i$, for $i=1,\cdots,4$ in the statement of the proposition
are  dominant. We have to prove that they equal $\rho(\hat
w\hat\nu)$, for some $\hat w$. This is checked by:
\begin{itemize}
\item $\nu^1=\rho(\hnu)$;
\item $\nu^2=\rho(\hs_1\hnu)=\rho(-\hnu_1,\hnu_1+\hnu_2,\hnu_3)$ \\
or $\nu^2=\rho(\hs_3\hnu)=\rho(\hnu_1,\hnu_2+\hnu_3,-\hnu_3)$;
\item $\nu^3=\rho(\hs_1\hs_2\hnu)=\rho(-\hnu_1-\hnu_2,\hnu_1,2\hnu_2+\hnu_3)$\\
or $\nu^3=\rho(\hs_3\hs_2\hnu)=\rho(\hnu_1+\hnu_2,\hnu_2+\hnu_3,-2\hnu_2-\hnu_3)$;
\item $\nu^4=\rho(\hs_1\hs_2\hs_3\hnu)=\rho(-\hnu_1-\hnu_2-\hnu_3,\hnu_1,2\hnu_2+\hnu_3)$\\
or $\nu^4=\rho(\hs_3\hs_2\hs_3\hnu)=\rho(\hnu_1+\hnu_2+\hnu_3,\hnu_2,-2\hnu_2-\hnu_3)$.
\end{itemize}
\end{proof}

\bigskip
 The set $\Delta_2$ is reduced to the short root $\alpha_1$.
Then Theorem~\ref{th:main} gives  line segments parallel to $\alpha_1$ and having (at least) one PRV component as end point. 

\begin{coro}\label{cor:G2}
The irreducible representation $V_{G_2}(\nu)$ occurs with multiplicity
at least one in $V_{{\rm Spin}_7}(\hnu)$ for all dominant weights $\nu$
contained in one of the three following line segments (distinct in general): 
\begin{itemize}
\item the line segment whose end points are $\nu^1$ and $\nu^2$, of length $\operatorname{min}(\hnu_1,\hnu_3)$;
\item the line segment whose end points are $\nu^3$ and
  $\nu^{3'}=(\hnu_1+3\hnu_2+\hnu_3,-\hnu_2)$, of length $\operatorname{min}(\hnu_1+\hnu_2,2\hnu_2+\hnu_3)$;
\item the line segment whose end points are $\nu^4$ and
  $\nu^{4'}=(\hnu_1+3\hnu_2+2\hnu_3,-\hnu_2-\hnu_3)$, of length $\operatorname{min}(\hnu_1+\hnu_2+\hnu_3,2\hnu_2+\hnu_3)$. 
 \end{itemize}
\end{coro}

Remark that $\nu^{3'}$ and $\nu^{4'}$ are not dominant in general and
are never strictly dominant. In particular the line segments with end points $\nu^3$ and $\nu^4$ intersect the wall of the dominant chamber corresponding to $\alpha_2$.

We illustrate this result in Figure~\ref{figure:G2} where $\hnu=(2,2,1)$.

\begin{figure}[h]
\begin{center}
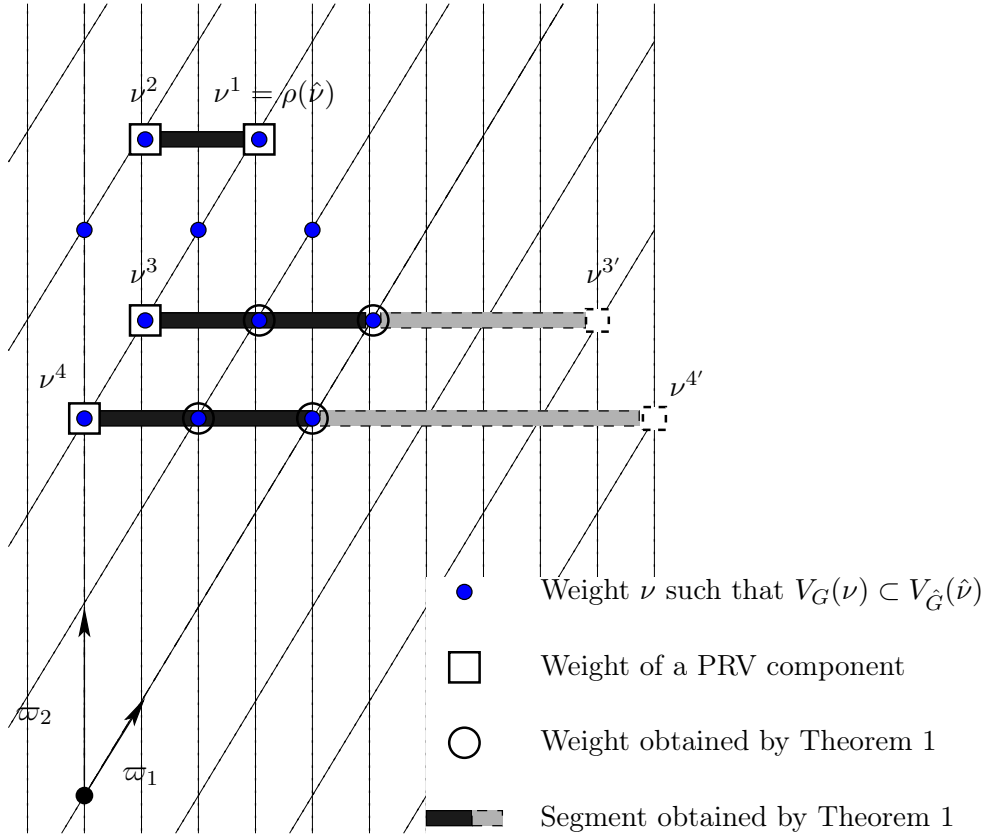
\caption{Illustration of Corollary \ref{cor:G2}}
\label{figure:G2}
\end{center}
\end{figure}

\begin{remark}
In that example, note that $\rho(\hnu)-\alpha_2$ is not a highest weight of a $G$-submodule of $V_\hG(\hnu)$. Nevertheless it is dominant and we have $\langle \hnu,\halpha_2\rangle=2$. Hence we cannot expect to extend our result to the roots $\alpha$ in $\Delta_1$.
\end{remark}

\section{A generalization of Theorem \ref{th:main}}
\label{sec:ortho}
In this section, we come back to the general situation of connected
reductive groups $G\subset \hG$ such that $\hG/G$ is spherical of  minimal rank. 

To state the result, we define the basic notion of hyperrectangle generated by a finite
set of pairwise orthogonal line segments having one end point in common.
For $1\leq k\leq s$, let
$I_k=[a,b_k]$ be $k$ pairwise orthogonal line segments in an affine space. The
hyperrectangle generated by the family $(I_k)_{1\leq k\leq s}$ is
$$R=\{a+\sum_{k=1}^s\lambda_k\overrightarrow{ab_k}\ |\ \lambda_k\in[0,1]\}.$$

We improve Theorem \ref{th:main}, by replacing the line segments $S_{\alpha,\hw,\hnu}$
by hyperrectangles containing them.

\begin{theo}\label{th:orth}
Let $\hnu$ be a dominant weight of $\hG$ and let $\hw\in\hW$ such that
$\rho(\hw\hnu)$ is a dominant weight of $G$. 
Let $\alpha^1,\dots,\alpha^s$ be $s$ pairwise orthogonal simple roots of $G$ in
$\Delta_2$ (with $s\geq 1$). Denote by
$R_{\alpha^1,\dots,\alpha^s,\hw,\hnu}$ the hyperrectangle generated by
the $s$ pairwise orthogonal line segments
$(S_{\alpha^i,\hw,\hnu})_{i\in\{1,\dots,s\}}$ (that have
$\rho(\hw\hnu)$ as common end point).
 
Then, for any $\nu\in R_{\alpha^1,\dots,\alpha^s,\hw,\hnu}\cap
X^+(T)$, the irreducible $G$-module $V_G(\nu)$ occurs with
multiplicity at least one in $V_\hG(\hnu)$.
\end{theo}

\begin{proof}
The proof of this theorem is very similar to the proof of Theorem
\ref{th:main}, so that we will only explain the main changes. 
First by a similar reasoning to that of Section \ref{sec:reduction},
the proof reduce to the following: if
$\nu=\rho(\hw\hnu)+\sum_{i=1}^sk^i\alpha^i$, where the integers $k^i$
satisfy three inequalities similar to inequalities~\eqref{eq:1}, then $V_G(\nu)$ occurs in $V_\hG(\hnu)$. 

For each $i\in\{1,\dots,s\}$, there exist two simple roots $\halpha^i_1$
and $\halpha^i_2$ such that
$\rho(\halpha^i_1)=\rho(\halpha^i_2)=\alpha^i$. Index them in such a way that, for any $i\in\{1,\dots,s\}$,
$$\l\hw\hnu,(\halpha^{i}_{1})^{\vee}\r\leq \l\hw\hnu,(\halpha_2^i)^\vee\r.$$  
Arguing like in Section~\ref{sec:proof:th:tech}, we may assume that for
$i\in\{1,\dots,s\}$, 
$$\l\hw\hnu,({\halpha^i_1})^\vee\r>0,$$
and hence that
$l(s_{\halpha_1^i}\hw)=l(\hw)+1$. 

Set $\hv=s_{\halpha_1^1}\ldots s_{\halpha_1^s}\hw$. Note that the
simple reflexions
$s_{\halpha_1^i}$ commute. Set $X=G/B\times \hG/\hB$, $S=\cap_{i=1}^s(\ker
\alpha^i)^\circ\subset T$ and let $C$ be the irreducible component of
$X^S$ containing $(w_0B/B,\hat vB/B)$. Observe that $G^S/S$ is the product of $s$ groups isomorphic to $\SL_2$ or ${\rm P\SL_2}$. Note also that $C$ is isomorphic to $((\PP^1)^3)^s$, with the natural action of $G^S/S$. In
particular, the inequalities satisfied by the integers $k^i$ implies that
$C^{ss}(\Mi_{|C},G^S)\neq\varnothing$. Then $X^{ss}(\Mi)\cap
C\neq\varnothing$. We deduce that there exists a positive integer $n$ and a
$G$-section $\overline\tau$ of $\Li_{n\hnu}$ which is a eigenvector of
weight $-n\nu$ for $B^-$ and whose restriction to $X_\hv$ is not
identically zero.

Since the simple roots $\alpha^i_1$ are pairwise orthogonal, for any
$0\leq l^i\leq \l\nu,\alpha^i\r$, the weight $\nu-\sum_i l^i\alpha^i$
of $T$ occurs in $V_H(\nu)$ with multiplicity exactly one. We deduce
(see Section~\ref{sec:asymptotic}) that $\dim
H^\circ(X_\hnu^\circ,\Li_\hnu)^{(T)_{\nu-\sum_i l^i\alpha^i}}=1$. We
conclude as in Section~\ref{sec:conclusion}. 
\end{proof} 

We illustrate Theorem \ref{th:orth} by the following example.

\begin{ex}\label{ex:orth}
Let $G=\operatorname{Sp}_6$ diagonally embedded in
$\hG=\operatorname{Sp}_6\times\operatorname{Sp}_6$. In that case, the
first and the third simple roots of $\operatorname{Sp}_6$ are orthogonal.
Consider the dominant weight $\hnu=((2,0,4),(2,0,2))$ in the basis of
fundamental weights. 
The irreducible components $V_G(\nu)$ of the tensor
product $V_\hG(\hnu)$, obtained by Theorem \ref{th:orth},  are
represented in Figure~\ref{fig:Sp6-1}. On this figure, the standard
basis is 
the  basis denoted by $(\epsilon_1,\epsilon_2,\epsilon_3)$ in Bourbaki.

\begin{figure}[!h]
 \centering
\begin{tikzpicture}
\node[%
minimum height=10.5cm,
minimum width=0.9\textwidth] (Frame) {\null};

\node[below] (image) at (Frame.north) {\includegraphics[height=8cm]{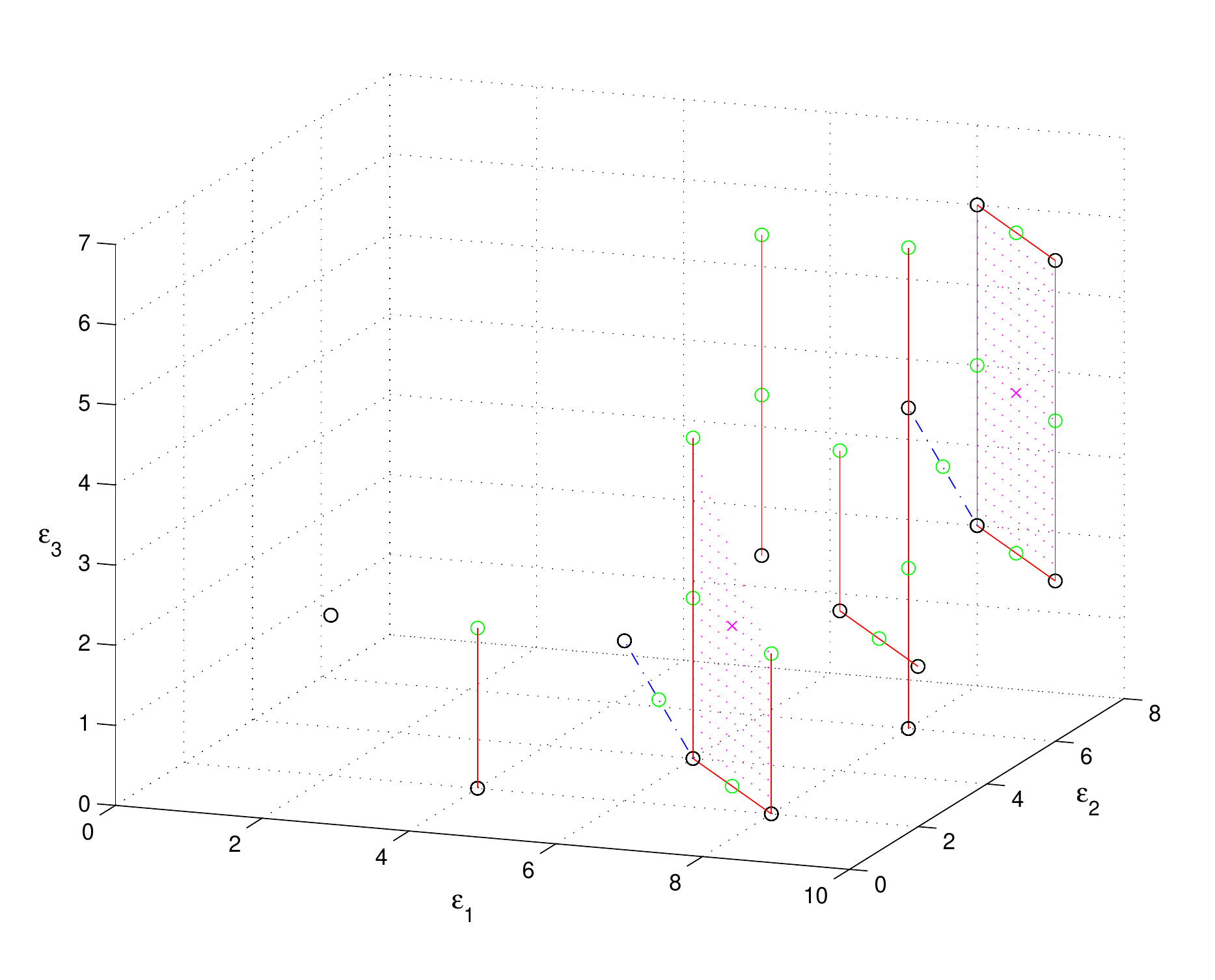}};
\node[draw,above] (legende) at (Frame.south) {
\begin{tabular}{cl}
  \begin{tikzpicture}    \draw (0,0) circle (0.06); \end{tikzpicture}
  & classical PRV component\\
\begin{tikzpicture}    \draw[green] (0,0) circle (0.06); \end{tikzpicture}
& component obtained with Theorem~\ref{th:main}\\
\begin{tikzpicture}    \draw[red] (-0.2,0.1) -- (0.2,0.1); \end{tikzpicture}
& line segments parallel to $\alpha_1$ or $\alpha_3$\\
\begin{tikzpicture}    \draw[blue,dashed] (-0.2,0.1) -- (0.2,0.1); \end{tikzpicture}
&line segments parallel to $\alpha_2$\\
\begin{tikzpicture}    \draw[magenta] (-0.06,0.06) -- (0.06,-0.06);
\draw[magenta] (-0.06,-0.06) -- (0.06,0.06);
 \end{tikzpicture}
&component obtained with Theorem~\ref{th:orth}
\end{tabular}
};
\end{tikzpicture}
 \caption{some weights of $V_{{\rm Sp}_6}(2,0,4)\otimes V_{{\rm Sp}_6}(2,0,2)$}
  \label{fig:Sp6-1}
\end{figure}
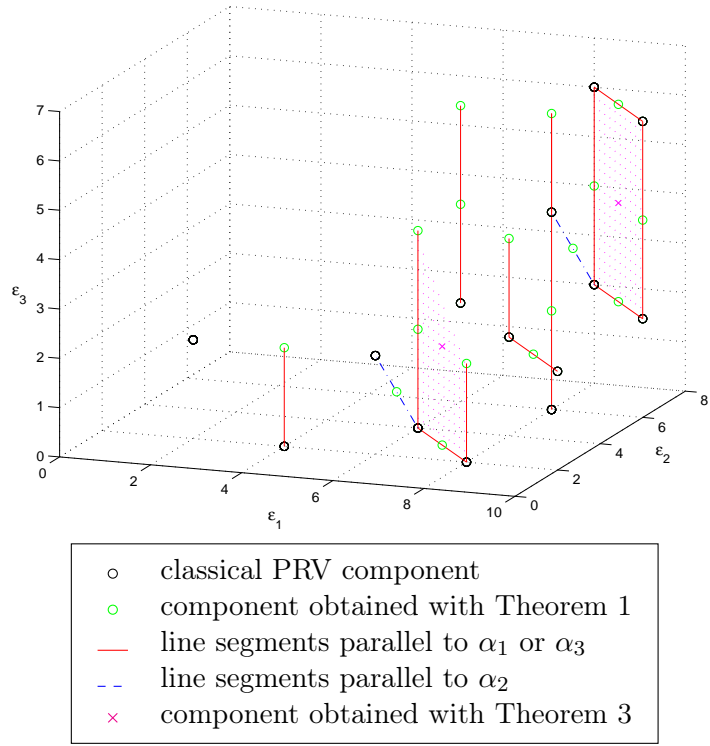

In this example, we observe that Theorem \ref{th:orth} gives two new
components in addition of those obtained by Theorem \ref{th:main}. In
fact, Theorem \ref{th:orth} gives here two rectangles, one of them
being truncated by the dominant chamber (see Figure~\ref{fig:Sp6-1}).
\end{ex}

\section{A specific result for tensor product decomposition}
\label{sec:sym}

From now on, we specialize to the case where $G$ is diagonally embedded in $\hG=G\times G$. 
Fix two dominant weights of $G$.
Then, Theorem~\ref{th:main} produces line segments in $X(T)\otimes\QQ$
parallel to any simple roots of $G$.
In this section, we use the symmetry induced by duality, in order to
produce similar line
segments whose direction is a root but not necessarily a simple one.

Let $\mu$ and $\nu$ be two dominant weights of $G$.
We say that a non-necessarily dominant weight $\lambda$ is a {\em virtual PRV component} (with respect
to $\mu$ and $\nu$) if there exists
$v$ and $w$ in $W$ such that $\lambda=v\mu+w\nu$.
Now we state the main result of this section. 

\begin{theo}\label{th:mainsym}
Let $v$ and 
 $w$ be two elements of $W$, and
let $\beta$ be a root of $G$. 

The line through the virtual PRV component
$\lambda_1=v\mu+w\nu$ in the direction $\beta$ contains the following four
virtual PRV components: $\lambda_1$, $\lambda_2= v\mu+s_\beta w\nu$,
$\lambda_3=s_\beta v\mu+w\nu$ and $\lambda_4=s_\beta v\mu+s_\beta
w\nu$. Choose $i$ and $j$ such that  $\langle \lambda_i,\,\beta^\vee\rangle$ and $\langle
\lambda_j,\,\beta^\vee\rangle$ are two largest  
among the family $(\langle
\lambda_l,\,\beta^\vee\rangle)_{l=1,2,3,4}$. 
Denote by $S_{\beta,v,w,\mu,\nu}$ the line segment in
$X(T)\otimes\QQ$ whose end points are $\lambda_i$ and $\lambda_j$.

Suppose that at least one of the following roots
  $\beta$, $v^{-1}\beta$ or $w^{-1}\beta$ is simple.  Then, for any
$\lambda\in
S_{\beta,v,w,\mu,\nu}\cap X^+(T)$, the irreducible $G$ module
  $V_G(\lambda)$ occurs with multiplicity at least one in $V_G(\mu)\otimes V_G(\nu)$.
\end{theo}

\begin{remark}
\begin{enumerate}
\item If $\langle
\lambda_i,\,\beta^\vee\rangle= \langle
\lambda_j,\,\beta^\vee\rangle$ then the
line segment is reduced to a point.
\item Note that in the case of a simple direction, if
  $S_{\beta,v,w,\mu,\nu}\cap X^+(T)$ is not empty, the line segment $S_{\beta,v,w,\mu,\nu}$ defined as above
always contains a PRV component (see \cite[Proposition 2]{MPR}). But
if the direction is not simple, there exist line
segments intersecting $X^+(T)$ and that contain no PRV component (see the example at the end
of the section).
\end{enumerate}
\end{remark}

As in the general case, we rewrite Theorem~\ref{th:mainsym} as follows.

\begin{theo}\label{th:techsym}
Let $\lambda$, $\mu$, and $\nu$ be three dominant weights of $G$, let $v$, $w$ be two
elements of $W$, and let $\beta$ be a root of $G$ such that one of
the following roots $\beta$, $v^{-1}\beta$ or $u^{-1}\beta$ is simple. Suppose
that $\lambda=v\mu+w\nu-k\beta$ where $k$ is an integer satisfying: 
\begin{subeqnarray}\label{eq:sym}
k & \geq & 0,\\
k & \leq & \langle v\nu,\,\beta^\vee\rangle,\\
k & \leq  & \langle w\mu,\,\beta^\vee\rangle.
\end{subeqnarray}
 
Then,
the irreducible $G$-module $V_G(\lambda)$ occurs with multiplicity at
least one in $V_G(\mu)\otimes V_G(\nu)$.
\end{theo}

\begin{proof}
The case where $\beta$ is simple is not new: it is the second result
of our preceding work \cite{MPR}, or the simple translation of
Theorem~\ref{th:tech}. 
The two other cases can be deduce from this one. Indeed, assume that
$\alpha=v^{-1}\beta$ is simple. Rewrite the condition
\begin{eqnarray}
\lambda=v\mu+w\nu-k\beta\label{eq:linesym}
\end{eqnarray}
as $\mu=v^{-1}\lambda-v^{-1}w\nu+k\alpha$, or equivalently as 
\begin{eqnarray}
\mu=v^{-1}\lambda+s_\alpha v^{-1}ww_0(-w_0\nu)+(k-\langle
v^{-1}w\nu,\,\alpha^\vee\rangle)\alpha.\label{eq:linesym3}
\end{eqnarray}
Setting $\nu'=-w_0\nu$,
$v'=v^{-1}$, $w'=s_\alpha v^{-1}ww_0$ and $k'=\langle
v^{-1}w\nu,\,\alpha^\vee\rangle-k$, equation (\ref{eq:linesym3})
becomes: 
\begin{eqnarray}
\mu=v'\lambda+w'\nu'-k'\alpha\,.\label{eq:linesym4}
\end{eqnarray}
One can check that conditions (\ref{eq:sym}) are equivalent to 
\begin{eqnarray*}
k'& \geq & 0,\\
k' & \leq & \langle v'\nu',\,\alpha^\vee\rangle,\\
k' & \leq & \langle w'\lambda',\,\alpha^\vee\rangle.
\end{eqnarray*}

Hence, Theorem~\ref{th:tech} or \cite[Theorem 2]{MPR} implies that
$V_G(\mu)$ appears in the tensor product $V_G(\lambda)\otimes
V_G(\nu')$.  But as $V_G(\nu')=V_G(-w_0\nu)=V_G(\nu)^\ast$,
we deduce that $V_G(\lambda)$ appears in $V_G(\mu)\otimes V_G(\nu)$.

\end{proof}


\begin{proof}[Proof of: Theorem~\ref{th:techsym} implies
    Theorem~\ref{th:mainsym}] 
By eventually changing $v$
by $s_\beta v$ and $w$ by $s_\beta w$, we
may assume that the two values $\langle v\nu,\,\beta^\vee\rangle$ and 
$\langle w\mu,\,\beta^\vee\rangle$ are non-negative. Assume also
that $\langle v\nu,\,\beta^\vee\rangle\geq\langle
w\mu,\,\beta^\vee\rangle$ so that two largest values among the family $(\langle
\lambda_l,\,\beta^\vee\rangle)_ {l=1,2,3,4}$ are $\langle
\lambda_1,\,\beta^\vee\rangle$ and $\langle \lambda_2,\,\beta^\vee\rangle$. 
A weight $\lambda$ belongs to the line segment
$S_{\beta,v,w,\mu,\nu}=[\lambda_1,\lambda_2]$ if and only if $\lambda=v\mu+w\nu-k\beta$ with
$0\leq k\leq \langle v\nu,\beta^\vee\rangle$ and hence we can conclude with
Theorem~\ref{th:techsym}. 
\end{proof}

The following theorem is obtained by combining
Theorem~\ref{th:orth} and Theorem~\ref{th:mainsym}.

\begin{theo}\label{th:orthsym}
Let $v$ and $w$ be two elements of the Weyl group, and
let $(\beta^i)_{1\leq i\leq s}$ be $s$ pairwise orthogonal roots. Suppose that there exists
$u\in\{1,v,w\}$ such that for any $i\in\{1,\dots,s\}$, $u^{-1}\beta^i$
is simple.

 Denote by $R_{\beta^1,\ldots,\beta^s,v,w,\mu,\nu}$ the
 hyperrectangle generated by the $s$ pairwise orthogonal line segments
 $(S_{\beta^i,v,w,\mu,\nu})_{i\in\{1,\ldots,s\}}$ defined in Theorem~\ref{th:mainsym}. 
 
Then for any $\lambda\in R_{\beta^1,\ldots,\beta^s,v,w,\mu,\nu} \cap
X^+(T)$, the irreducible $G$-module $V_G(\lambda)$ occurs with
multiplicity at least one in $V_G(\mu)\otimes V_G(\nu)$.
\end{theo}

\begin{ex}
Consider the tensor product $V_{\operatorname{Sp}_6}(2,1,0)\otimes
V_{\operatorname{Sp}_6}(0,3,1)$. The irreducible components of this tensor product
obtained by Theorem \ref{th:mainsym} are represented in Figure~\ref{fig:Sp6-2}.

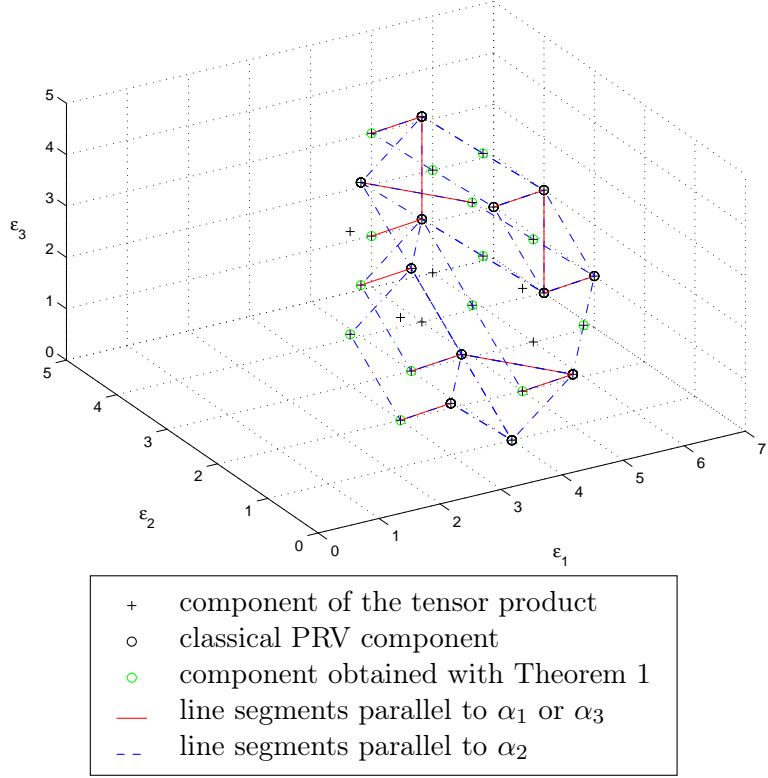
\begin{figure}[!h]
  \centering
  
  \begin{tikzpicture}
\node[%
minimum height=10.7cm,
minimum width=0.9\textwidth] (Frame) {\null};

\node[below] (image) at (Frame.north) {\includegraphics[height=8cm]{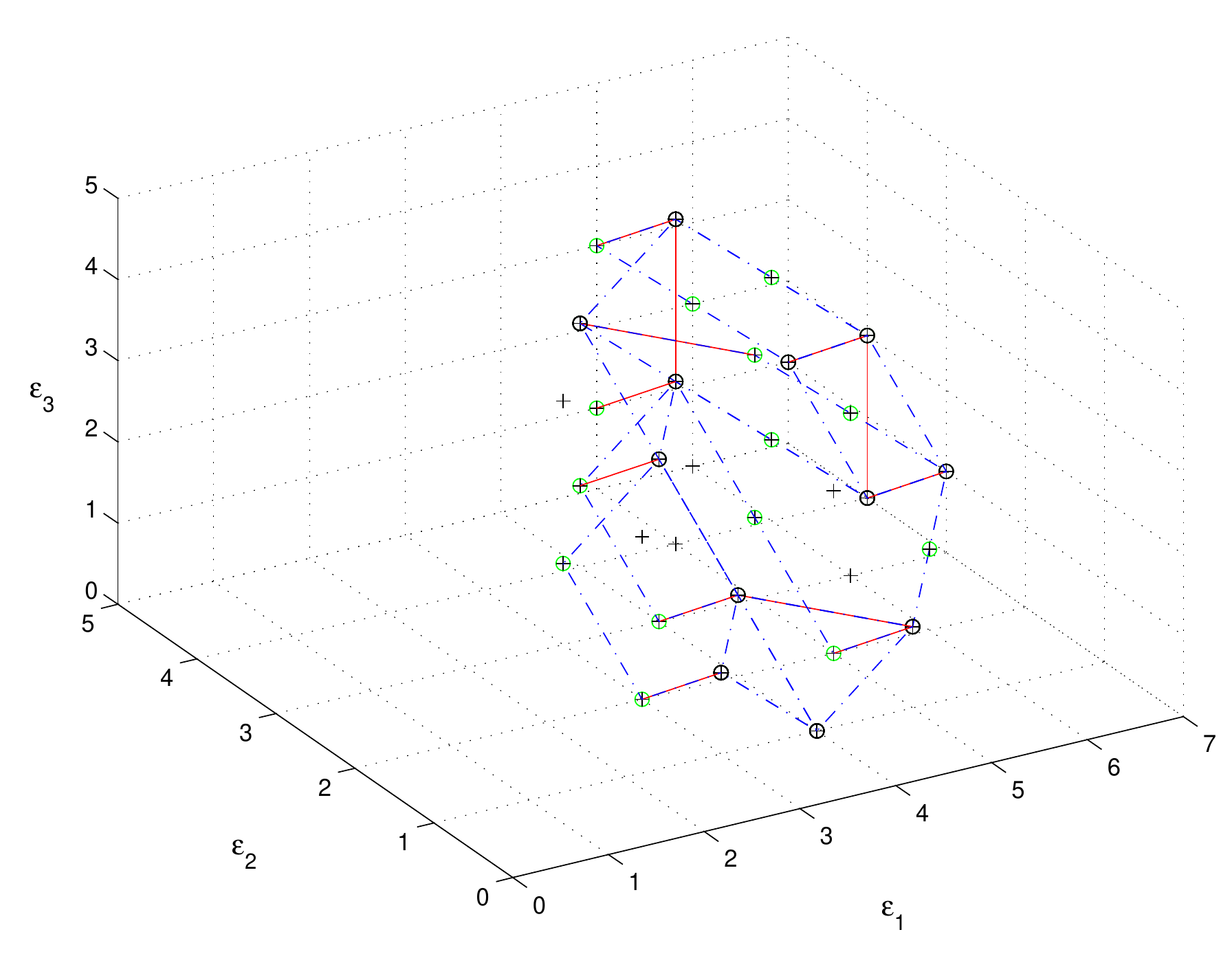}};
\node[draw,above] (legende) at (Frame.south) {
\begin{tabular}{cl}
  \begin{tikzpicture}
\draw (-0.06,0) -- (0.06,0);
\draw (0,-0.06) -- (0,0.06);
\end{tikzpicture}
& component of the tensor
product\\
  \begin{tikzpicture}    \draw (0,0) circle (0.06); \end{tikzpicture}
  & classical PRV component\\
\begin{tikzpicture}    \draw[green] (0,0) circle (0.06); \end{tikzpicture}
& component obtained with Theorem~\ref{th:main}\\
\begin{tikzpicture}    \draw[red] (-0.2,0.1) -- (0.2,0.1); \end{tikzpicture}
& line segments parallel to $\alpha_1$ or $\alpha_3$\\
\begin{tikzpicture}    \draw[blue,dashed] (-0.2,0.1) -- (0.2,0.1); \end{tikzpicture}
&line segments parallel to $\alpha_2$
\end{tabular}
};
\end{tikzpicture}

\caption{Weights of $V_{{\rm Sp}_6}(2,1,0)\otimes V_{{\rm Sp}_6}(0,3,1)$}
\label{fig:Sp6-2}
\end{figure}


There are 12 PRV
components, 
7 more weights given by Theorem~\ref{th:main}, 7 more weighs given by
Theorem~\ref{th:mainsym}, and no more with
Theorem~\ref{th:orthsym}. To summarize, we obtain 26 weights
(represented in Figure \ref{fig:Sp6-2}) among the
32 existing ones.

We observe a new phenomenon that cannot occur with Theorem
\ref{th:main}: the two end points of a line segment can be both outside
the dominant chamber (for example, the one in the lower left in Figure~\ref{fig:Sp6-2}).

Return to the example $V_{{\rm Sp}_6}(2,0,4)\otimes V_{{\rm
    Sp}_6}(2,0,2)$ seen in Example~\ref{ex:orth}. In this case, there are 14 PRV components, plus 17 new dominant weights with
 Theorem~\ref{th:main}, plus 2 new ones with Theorem \ref{th:orth},
 plus 16 new ones with Theorem \ref{th:mainsym}, plus 9 new ones with
Theorem \ref{th:orthsym}. 
To summarize, the results of this paper give 58 dominant weights among
the 83 existing ones. 
\end{ex}

\section{Appendix}

By lack of reference, 
we collect here technical results.

We begin with a result on the length of Weyl group elements.

\begin{lemma}\label{lemma:length}
Let $G$ be a reductive group. Fix a Borel subgroup containing a maximal torus. Let $W$ be the Weyl group of $G$. For any simple root $\alpha$, denote by $s_\alpha$ the corresponding reflection in $W$.
 
Let $\lambda$ be a dominant weight and $u$ be an element of $W$. Let $\beta$ be a simple root such that $\langle u\lambda,\beta^\vee\rangle>0$. Then $l(s_\beta u)=l(u)+1$. In particular we also have that if $\langle u\lambda,\beta^\vee\rangle<0$ then $l(u)=l(s_\beta u)+1$.
\end{lemma}

\begin{proof}
The length of $u$ equals the length of $u^{-1}$. It is the cardinality of $$\{\gamma \mbox{ positive root }\mid u^{-1}\gamma\mbox{ is a negative root}\}.$$
Let $\gamma$ be a positive root such that $u^{-1}\gamma$ is a negative root. We claim that $\gamma$ cannot be the simple root $\beta$. Indeed, since $\langle\lambda,u^{-1}\beta^\vee\rangle=\langle u\lambda,\beta^\vee\rangle>0$ and $\lambda$ is dominant, $u^{-1}\beta$ is positive. Then the application 
$$\begin{array}{ccc}
\{\gamma >0\mid u^{-1}\gamma<0\} & \longrightarrow & \{\gamma>0\mid u^{-1}s_\beta\gamma<0\}\\
\gamma & \longmapsto & s_\beta\gamma
\end{array}
$$
is well-defined because $u^{-1}s_\beta s_\beta\gamma=u^{-1}\gamma$ and
$s_\beta\gamma$ is a positive root for all positive roots $\gamma$
different from $\beta$. This application is also clearly injective, so
that the cardinality of the second set, which is the length of
$u^{-1}s_\beta$ and also $s_\beta u$, is bigger than the length of $u$.

But we have either $l(s_\beta u)=l(u)+1$ or $l(s_\beta u)=l(u)-1$ because $s_\beta$ is a simple reflection, so the proposition follows.
\end{proof}

Now, we give a result on the maximal torus of a spherical reductive subgroup.

\begin{lemma}\label{lemma:centra}
Suppose that $G$ is a spherical reductive subgroup  in a reductive group $\hG$.
Let $T$ be a maximal torus of $G$. Then, the centralizer of $T$ in $\hG$
is a maximal torus of $\hG$.
\end{lemma}

\begin{proof}
By \cite{Br:BGHfini}, since $G$ is spherical, it has finitely many
orbits in $\hG/\hB$. Since $G$ is reductive, $T$ has finitely many
fixed points in any $G$-orbit. Hence $T$ has finitely many fixed points in $\hG/\hB$. It follows that $T$ is regular in $\hG$.
\end{proof}

As a consequence of the preceding corollary, we have the following result.

\begin{coro}\label{coro:weyl}
In the situation of the preceding lemma, the Weyl group of $G$ is
canonically a subgroup of the Weyl group of $\hG$. 
\end{coro}

\begin{proof}
The lemma implies that the normalizer of $T$ in $G$ normalizes $\hat T$.
In particular, it is contained in the normalizer of $\hat T$ in $\hat G$.
The injection of the corollary is obtained by taking the quotient of this 
inclusion by $\hat T$.
\end{proof}

\bibliographystyle{amsalpha}
\bibliography{prv2}

\bigskip
\noindent P.L. Montagard {\tt pierre-louis.montagard@math.univ-montp2.fr}\\
Universit{\'e} Montpellier II - CC 51-Place Eug{\`e}ne Bataillon - 34095 Montpellier Cedex 5 - France\\

\noindent B. Pasquier {\tt boris.pasquier@math.univ-montp2.fr}\\
Universit{\'e} Montpellier II - CC 51-Place Eug{\`e}ne Bataillon - 34095 Montpellier Cedex 5 - France\\

\noindent N. Ressayre {\tt nicolas.ressayre@math.univ-lyon1.fr}\\
Institut Camille Jordan - Universit{\'e} Lyon 1 - 69622 Villeurbanne cedex -France
\end{document}

%% file: G2.pdf_t
\begin{picture}(0,0)%
\includegraphics{G2.pdf}%
\end{picture}%
\setlength{\unitlength}{4144sp}%
\begingroup\makeatletter\ifx\SetFigFont\undefined%
\gdef\SetFigFont#1#2#3#4#5{%
  \reset@font\fontsize{#1}{#2pt}%
  \fontfamily{#3}\fontseries{#4}\fontshape{#5}%
  \selectfont}%
\fi\endgroup%
\begin{picture}(5392,4986)(214,-7510)
\put(901,-7171){\makebox(0,0)[lb]{\smash{{\SetFigFont{11}{13.2}{\rmdefault}{\mddefault}{\updefault}{\color[rgb]{0,0,0}$\varpi_1$}%
}}}}
\put(271,-6811){\makebox(0,0)[lb]{\smash{{\SetFigFont{11}{13.2}{\rmdefault}{\mddefault}{\updefault}{\color[rgb]{0,0,0}$\varpi_2$}%
}}}}
\put(3376,-6541){\makebox(0,0)[lb]{\smash{{\SetFigFont{11}{13.2}{\rmdefault}{\mddefault}{\updefault}{\color[rgb]{0,0,0}Weight of a PRV component}%
}}}}
\put(3376,-6991){\makebox(0,0)[lb]{\smash{{\SetFigFont{11}{13.2}{\rmdefault}{\mddefault}{\updefault}{\color[rgb]{0,0,0}Weight obtained by Theorem \ref{th:main}}%
}}}}
\put(3376,-7441){\makebox(0,0)[lb]{\smash{{\SetFigFont{11}{13.2}{\rmdefault}{\mddefault}{\updefault}{\color[rgb]{0,0,0}Segment obtained by Theorem \ref{th:main}}%
}}}}
\put(3376,-6091){\makebox(0,0)[lb]{\smash{{\SetFigFont{11}{13.2}{\rmdefault}{\mddefault}{\updefault}{\color[rgb]{0,0,0}Weight $\nu$ such that $V_G(\nu)\subset V_{\hat{G}}(\hat\nu)$}%
}}}}
\put(946,-3121){\makebox(0,0)[lb]{\smash{{\SetFigFont{11}{13.2}{\rmdefault}{\mddefault}{\updefault}{\color[rgb]{0,0,0}$\nu^2$}%
}}}}
\put(1441,-3121){\makebox(0,0)[lb]{\smash{{\SetFigFont{11}{13.2}{\rmdefault}{\mddefault}{\updefault}{\color[rgb]{0,0,0}$\nu^1=\rho(\hat{\nu})$}%
}}}}
\put(946,-4201){\makebox(0,0)[lb]{\smash{{\SetFigFont{11}{13.2}{\rmdefault}{\mddefault}{\updefault}{\color[rgb]{0,0,0}$\nu^3$}%
}}}}
\put(406,-4831){\makebox(0,0)[lb]{\smash{{\SetFigFont{11}{13.2}{\rmdefault}{\mddefault}{\updefault}{\color[rgb]{0,0,0}$\nu^4$}%
}}}}
\put(3646,-4201){\makebox(0,0)[lb]{\smash{{\SetFigFont{11}{13.2}{\rmdefault}{\mddefault}{\updefault}{\color[rgb]{0,0,0}$\nu^{3'}$}%
}}}}
\put(4141,-4876){\makebox(0,0)[lb]{\smash{{\SetFigFont{11}{13.2}{\rmdefault}{\mddefault}{\updefault}{\color[rgb]{0,0,0}$\nu^{4'}$}%
}}}}
\end{picture}%